\documentclass[preprint]{elsarticle}
\usepackage{lineno}

\journal{}

\usepackage{slashbox}
\usepackage{amsmath}
\usepackage{amssymb}
\usepackage{amsthm}
\usepackage{graphicx}
\usepackage{epic,eepic,epsfig}
\usepackage{color}
\usepackage{subfigure}  
\usepackage{placeins}   
\usepackage{multirow}
\usepackage{epstopdf}
\usepackage{varwidth}
\usepackage{float}
\usepackage[table]{xcolor}

\newtheorem{theorem}{Theorem}[section]

\newtheorem{lemma}{Lemma}[section]
\newtheorem{remark}{Remark}[section]

\newtheorem{assumption}{Assumption}[section]

\definecolor{tabclr}{cmyk}{0,0,1,0}

\begin{document}

\title{Discontinuous Galerkin methods for semilinear elliptic boundary value problem}

\author[SCNU]{Jiajun Zhan}
\ead{zhan@m.scnu.edu.cn}

\author[SCNU]{Liuqiang Zhong}
\ead{zhong@scnu.edu.cn}

\author[SCNU]{Jie Peng\corref{cor}}
\ead{pengjie18@m.scnu.edu.cn}

\cortext[cor]{Corresponding author}
\address[SCNU]{School of Mathematical Sciences, South China Normal University, Guangzhou 510631, China}

\begin{abstract}
A discontinuous Galerkin (DG) scheme for solving semilinear elliptic problem is developed and analyzed in this paper.
The DG finite element discretizations are established, and the corresponding existence and uniqueness theorem is proved by using Brouwer's fixed point method.
Some optimal priori error estimates under both DG norm and $L^2$ norm are presented.
Numerical results are also shown to confirm the efficiency of the proposed approach.
\end{abstract}


\maketitle


\section{Introduction}

Given a bounded polygonal domain $\Omega \subset \mathbb{R}^d (d=2, 3)$ with the boundary $\partial \Omega$.
We consider the following semilinear equation
\begin{equation}\label{Eqn:u}
\left\{\begin{aligned}
- \Delta u &= f(x, u), \quad &&\mbox{in $\Omega$}, \\
 u &=0, \quad &&\mbox{on $\partial \Omega$. }
\end{aligned}\right.
\end{equation}
For simplicity, we will replace $f(x,u)$ with $f(u)$ in the following exposition.
The semilinear equation \eqref{Eqn:u} is widely used in many practical applications,
such as describing the  potential of a stable fluid or the stable temperatrue field with a source (See \cite{YouX03:479Chn}).

Discontinuous Galerkin (DG) methods are widely used numerical methodologies for the numerical solutions of partial differential equations.
They have many advantages in contrast to the conforming finite element methods (FEMs).
For example, DG methods allow more fexibility in handling equations whose types change within the computational domain and
the corresponding finite element space has no continuity constraints across the edges/faces of the triangulation.
Because of these advantages, DG methods are extended to various model problems,
such as elliptic problems \cite{ArnoldBrezzi02:1749}, Navier-Stokes equations \cite{BassiRebay97:267}, Maxwell equations \cite{ChungYuen14:613} and so on.
Recent years, some different types of DG methods have been developed,
such as the symmetric interior penalty discontinuous Galerkin (SIPDG) method \cite{Wheeler78:152},  incomplete interior penalty discontinuous Galerkin (IIPDG) method \cite{SunWheeler05:195}, local DG method \cite{CockburnShu98:2440} and so on.

In this paper, the semilinear elliptic boundary value problem is solved by SIPDG method.
First of all, the DG scheme of problem \eqref{Eqn:u} is given, the existence and uniqueness of the finite element solution of the DG scheme is derived by making use of Brouwer fixed point theorem.
Then, the optimal priori error estimates are given under DG norm and $L^2$ norm.
Finally, numerical results are shown to verify the theoretical findings.

To avoid the repeated use of generic but unspecied constants, we shall use $x \lesssim y $ to denote $x \leq C y$, where constant $C$ is a positive constant independent of the variables that appear in the inequalities and especially the mesh parameters. The notation $C_i$, with subscript, denotes specific important constant.

The rest of the paper is organized as follows.
In Section \ref{Cha:3}, the DG method is introduced for the semilinear elliptic problem \eqref{Eqn:u}, and it is proved that the discrete system has unique solution.
In Section \ref{Sec:error}, the optimal prior error estimates under DG norm and $L^2$ norm are provied.
In Section \ref{Cha:5}, the numerical experiments are presented.

\section{DG method}\label{Cha:3}
\setcounter{equation}{0}

In this section, a DG method and the corresponding well-posedness
will be introduced for the semilinear elliptic problem \eqref{Eqn:u}. To this end,
the coresponding Sobolev space and some assumptions need to be introduced firstly.
\subsection{Sobolev spaces and assumptions}

Given domain $S \subset \mathbb{R}^d(d=2, 3)$, for any integer $m \geq 0$, $p \geq 1$, we denote $W^{m, p}(S)$ as the standard Sobolev space with norm $\| \cdot \|_{m, p, S}$.
For simplicity of notation, we denote $H^m(S) = W^{m, 2}(S)$ and $\| \cdot \|_{m, S} = \| \cdot \|_{m, 2, S}$.
Especially when $S= \Omega$, we denote $\| \cdot \|_{m} = \| \cdot \|_{m, 2, \Omega}$, and $H^1_0(\Omega) := \{u\in H^1(\Omega) : u|_{\partial \Omega} =0 \}$.

Semilinear elliptic boundary value problem  \eqref{Eqn:u} generally have multiple solutions (See \cite{ChenCMXieZQ08:42}).
We assume that problem \eqref{Eqn:u} has at least one solution $u \in H_0^1(\Omega) \cap H^{r+1}(\Omega)\ (r\geq 1)$ (See \cite{XuJC94:231, XuJC94:79, YouX03:479Chn}).
In order to analyze the existence and uniqueness of the finite element solution $u_h$
, the following assumptions should be made for the function $f(u)$ (See also Assumption 2.1 of \cite{ZhangWFFanRH20:522}):
\begin{assumption}\label{Assf}
Nonlinear term $f(u)$ satisfies
\begin{equation} \label{Eqn:Assf}
f_u(u)\leq 0,~f_u(u)~\mbox{and}\
f_{uu}(u)\ \mbox{are bounded}.
\end{equation}
\end{assumption}


\subsection{Discontinuous finite element space}

Let $\mathcal{T}_h$ be a quasi-uniform family of partitions of $\Omega$ into $d$-dimensional simplices $K$(triangles if $d=2$ and tetrahedra if $d=3$), where $h:= \max_{K \in \mathcal{T}_h}\{h_K\}$, $h_K$ is the circumscribed circle diameter of element $K \in \mathcal{T}_h$.
We assume that $\mathcal{T}_h$ is conforming which mean that it doesn't contain hanging nodes.

We introduce discontinuous Sobolev space $H^{s}(\mathcal{T}_h)$ in $\mathcal{T}_h$, which is defined as
\begin{equation} \label{Eqn:Hs}
H^{s}(\mathcal{T}_h) = \{ v \in L^2(\Omega) : v|_K \in H^s(K),\ \forall\  K \in \mathcal{T}_h\}, \ \  s \geq 1,
\end{equation}
and the corresponding discontinuous finite element space $V_h$ on $\mathcal{T}_h$ is defined as
\begin{equation}\label{Eqn:Vh}
V_h = \{ v_h \in L^2(\Omega) : v_h|_K \in \mathcal{P}_r(K),\ \forall\  K \in \mathcal{T}_h, v_h|_{\partial \Omega}=0 \},
\end{equation}
where $\mathcal{P}_r(K)$ is the set of polynomials of degree at most $r$ on $K$.

Let $\mathcal{E}_h$ be the set of all edges or faces, $\mathcal{E}_h^0$ be the set of all interior edges or faces, $\mathcal{E}_h^\partial =\mathcal{E}_h \backslash \mathcal{E}_h^0$ be the set of all boundary edges or faces.
Let $e  \in \mathcal{E}_h^0$ be an interior edge or face shared by two elements $K_{\pm} \in \mathcal{T}_h$.
We define averages and jumps of scalar function $v \in H^s(\mathcal{T}_h)$ and vector function $\boldsymbol{w} \in [H^s(\mathcal{T}_h) ]^2$ on $e$ by
$$
\begin{aligned}
&[v] = v_+ \boldsymbol{n}_+ + v_- \boldsymbol{n}_-,  &&\{v\}  = \dfrac{1}{2} (v_+ + v_-),
\\
&[\boldsymbol{w}] = \boldsymbol{w}_+ \boldsymbol{n}_+ + \boldsymbol{w}_- \boldsymbol{n}_-,  &&\{\boldsymbol{w}\}  = \dfrac{1}{2} (\boldsymbol{w}_+ + \boldsymbol{w}_-),
\end{aligned}
$$
where $v_\pm = v|_{K_\pm \cap e}$, $\boldsymbol{w}_\pm = \boldsymbol{w}|_{K_\pm\cap e}$ and $\boldsymbol{n}_\pm$ be the unit normal of $e$ pointing towards the outside of $K_\pm$.

For a boundary edge or face $e \in \mathcal{E}_h^\partial $, we define averages and jumps of scalar function $v \in H^s(\mathcal{T}_h)$ and vector function $\boldsymbol{w} \in [H^s(\mathcal{T}_h) ]^2$ on $e$ by
$$
\begin{aligned}
&[v] = v \boldsymbol{n}, && \{v\}  = v,
\\
&[\boldsymbol{w}] = \boldsymbol{w} \boldsymbol{n}, && \{\boldsymbol{w}\}  = \boldsymbol{w}.
\end{aligned}
$$
where $\boldsymbol{n}$ is the unit normal of $e$ pointing towards the outside of $\Omega$.

\subsection{Discontinuous finite element method}

In this subsection, a discontinuous finite element discrete scheme of the problem \eqref{Eqn:u} is first given, then a series of preparatory lemmas are given, and finally the discrete variational problem is proved to have an unique solution.

\subsubsection{Discrete scheme}

The discontinuous finite element discrete form of \eqref{Eqn:u} is given: Find $u_h \in V_h$, such that
\begin{equation}\label{Eqn:Varuh}
a_h(u_h, v_h) = (f(u_h), v_h), \quad \forall\  v_h \in V_h,
\end{equation}
where
\begin{eqnarray} \nonumber
a_h(w_h, v_h) &=& \sum\limits_{K \in \mathcal{T}_h} (\nabla w_h, \nabla v_h)_K - \sum_{e \in \mathcal{E}_h} \int_e \{\nabla w_h\} \cdot [v_h] \mathrm{d}s
\\  \label{Eqn:ah}
 &&  - \sum_{e \in \mathcal{E}_h} \int_e \{\nabla v_h\} \cdot [w_h] \mathrm{d}s + \sum_{e \in \mathcal{E}_h}  \int_e \dfrac{\lambda}{h_e}  [w_h]\cdot [v_h] \mathrm{d}s,\\\nonumber
(f(u_h), v_h) &=& \sum\limits_{K \in \mathcal{T}_h} (f(u_h), v_h)_K,
\end{eqnarray}
here the constant $\lambda > 0$ is a penalty parameter..

We define DG norm $|\| \cdot \||_h$ by
\begin{equation}\label{Eqn:|||}
|\| w \||_h = \left(  \sum\limits_{K \in \mathcal{T}_h} \| \nabla w\|_{0, K}^2
+ \sum_{e \in \mathcal{E}_h}   \dfrac{h_e}{\lambda} \|\{\nabla w\}\|_{0, e}^2
+  \sum_{e \in \mathcal{E}_h}  \dfrac{\lambda}{h_e} \|[w]\|_{0, e}^2 \right)^{\frac{1}{2}}.
\end{equation}

\subsubsection{Preliminaries}

Some Preliminaries needs to be introduced before the existence and uniqueness of the problem \eqref{Eqn:Varuh} is proved. The first lemma gives the continuity and coerciveness of the bilinear form $a_h(\cdot, \cdot)$.
\begin{lemma}[\cite{ZhangT12Book}, Lemma 2.4]\label{Lem:ahleqgeq}
The bilinear form $a_h(\cdot, \cdot)$ defined in \eqref{Eqn:ah} has the following properties
\begin{equation}\label{Eqn:ahleq}
|a_h(w, v)| \leq 3 |\|u\||_h \cdot |\| v\||_h, \quad \forall\  w, v \in H^{1+s}(\mathcal{T}_h), s \geq1/2,
\end{equation}
and when $\lambda$ is sufficiently large, it is obtained that
\begin{equation}\label{Eqn:ahgeq}
\dfrac{1}{4} |\| v_h\||_h^2 \leq a_h(v_h, v_h), \quad \forall\  v_h \in V_h.
\end{equation}
\end{lemma}
\begin{remark}
To ensure that the solution of the problem \eqref{Eqn:Varuh} exists uniquely, this paper always assumes that the penalty parameter $\lambda$ is sufficiently large (See Section \ref{Sec:well-posed}). Therefore Lemma \ref{Lem:ahleqgeq} is always true.
\end{remark}

Next lemma provides a quite standard tool in discontinuous finite element analysis.
\begin{lemma}[\cite{ZhangT12Book}, Lemma 2.1] \label{Lem:int}
Let $(v, \boldsymbol{w}) \in H^{s}\left(\mathcal{T}_{h}\right) \times\left[H^{s}\left(\mathcal{T}_{h}\right)\right]^{d} (s \geq 1)$, and the following identity holds
$$
\sum_{K \in \mathcal{T}_{h}} \int_{\partial K} v \boldsymbol{w} \cdot \boldsymbol{n} \mathrm{d} s=\sum_{e \in \mathcal{E}_h} \int_e \{\boldsymbol{w}\} \cdot[v] \mathrm{d} s+\sum_{e \in \mathcal{E}_h^0} \int_e \{v\}[\boldsymbol{w}] \mathrm{d} s.
$$
Further, let $u \in H^{s}\left(\omega_{e}\right) (s \geq 1)$, we have
$$
\int_{e} [u] v d s=0, \quad \forall  v \in L_{2}(e), e \in \mathcal{E}_h^0,
$$
where $\omega_{e}$ is the union of $K_+$ and $K_-$ which share $e$.
\end{lemma}

We will also need the following trace ineuqality and inverse ineuqality.
\begin{lemma}[\cite{RiviereWheeler01:902}, (2.5)]\label{Lem:trace}
Let $e$ denote an edge ($d=2$) or a face ($d=3$) of element $K \in \mathcal{T}_h$, then there exist a positive constant $C_t$ independent of $h$ such that
\begin{equation}\label{Eqn:trace}
\|v\|_{0, e}^2 \leq C_t \left(h_e^{-1} \| v\|_{0, K}^2 + h_e |v|_{1, K}^2 \right), \quad \forall v \in H^1(K).
\end{equation}
\end{lemma}

\begin{lemma}[\cite{Ciarlet02Book}, Theorem 3.2.6]\label{Lem:inverse}
Suppose mesh generation $\mathcal{T}_h$ is quasi-uniform, then for any $1 \leq q \leq \infty, 0 \leq l \leq m$, there exist a positive constant $C_s$ independent of $h$ such that
\begin{equation}\label{Eqn:inverse}
|v_h|_{m, q, K} \leq C_s h^{ \frac{d}{q} -  \frac{d}{2}} h^{l-m} |v_h|_{l, 2, K}, \quad \forall K \in \mathcal{T}_h, v_h \in V_h.
\end{equation}
\end{lemma}

Next lemma shows that the $L^2$ norm can be controlled by the DG norm in $V_h$.
The proof of this lemma is similar to the proof of Lemma 2.2  in \cite{ZhangT12Book}.
\begin{lemma}\label{Lem:L2leq|||}
For any $v_h \in V_h$, we have
\begin{equation}\label{Eqn:L2leq|||}
\|v_h \|_0 \lesssim |\| v_h\||_h.
\end{equation}
\end{lemma}

\begin{proof}
For any fixed $g \in L^2(\Omega)$, we introduce the following auxiliary problem:
Find $w \in H^2(\Omega)$ such that
\begin{equation*}
\left\{\begin{aligned}
- \Delta w &=g, &&x \in \Omega, \\
w&=0, &&x \in \partial \Omega,
\end{aligned}\right.
\end{equation*}
and we assume that the following regularity result holds
\begin{equation}\label{Eqn:w2leqg0}
\|w\|_2 \lesssim \|g\|_0.
\end{equation}

Let $v_h \in V_h$, using the integration by parts, Lemma \ref{Lem:int}, $w\in H^2(\Omega) $ and Cauchy inequality, we have
\begin{eqnarray}
(v_h, g) &=&\sum\limits_{K \in \mathcal{T}_h}  (\nabla w, \nabla v_h)_K - \sum\limits_{K \in \mathcal{T}_h} \int_{\partial K} \dfrac{\partial w}{\partial \boldsymbol{n}} v_h \mathrm{d}s \nonumber
\\ \nonumber
&=& \sum\limits_{K \in \mathcal{T}_h} (\nabla w, \nabla v_h)_K - \sum_{e\in \mathcal{E}_h} \int_{e} \{\nabla w\} [v_h] \mathrm{d}s
%
\\ \nonumber
&\leq&
\sum\limits_{K \in \mathcal{T}_h}  \| \nabla v_h\|_{0, K} \|\nabla w\|_{0, K} +
\left( \sum_{e\in \mathcal{E}_h} \dfrac{h_e}{\lambda} \| \{\nabla w\}\|_{0, e}^2 \right)^{1/2} \left( \sum_{e\in \mathcal{E}_h} \dfrac{\lambda}{h_e} \| [v_h]\|_{0, e}^2 \right)^{1/2}
%
\\ \nonumber
&\leq&
\left(  \sum\limits_{K \in \mathcal{T}_h} \|\nabla w\|_{0, K}^2 + \sum_{e \in \mathcal{E}_h} \dfrac{h_e}{\lambda} \| \{\nabla w\}\|_{0, e}^2 \right)^{1/2}
\left( \sum\limits_{K \in \mathcal{T}_h} \| \nabla v_h\|_{0, K}^2 + \sum_{e\in \mathcal{E}_h} \dfrac{\lambda}{h_e} \| [v_h]\|_{0, e}^2 \right)^{1/2}
%
\\ \label{Eqn:vhg}
&\lesssim &
\left( \sum\limits_{K \in \mathcal{T}_h} \|\nabla w\|_{0, K}^2 + \sum_{e \in \mathcal{E}_h} \dfrac{h_e}{\lambda} \| \{\nabla w\}\|_{0, e}^2 \right)^{1/2} |\|v_h\||_h.
\end{eqnarray}

Using trace inequality and shape-regularity of $\mathcal{T}_h$, it is obtained that
\begin{equation*}
\int_{e \cap \partial K} \dfrac{h_e}{\lambda} |\{\nabla w\}|^2 \mathrm{d}s \lesssim h_e \left(h_K^{-1}\|\nabla w\|_{0, K}^2 + h_K |w|_{2, K}^2\right) \lesssim \|w\|_{2, K}^2.
\end{equation*}

Substituting the above equation into \eqref{Eqn:vhg} and using \eqref{Eqn:w2leqg0}, we have
\begin{equation*}
(v_h, g) \lesssim \|w\|_{2} |\|v_h\||_h \lesssim \|g\|_0 |\|v_h\||_h, \quad \forall g \in L^2(\Omega).
\end{equation*}

Finally, from the above inequality and notice that $g\in L^2(\Omega)$ is arbitrary, the inequality \eqref{Eqn:L2leq|||} can be obtained.
\end{proof}

\begin{lemma}\label{Lem:inversepromote}
Suppose that the mesh generation $\mathcal{T}_h$ is quasi-uniform, then for any $v_h \in V_h$, we have
\begin{equation}\label{Eqn:inversepromote}
\|v_h\|_{0, \infty} \lesssim (\ln  h)^{1/2} |\|v_h\||_h.
\end{equation}
\end{lemma}

\begin{proof}
By using inequality (28) of \cite{brenner04:42}, it can be seen that
$$
\|v_h\|_{0, \infty}^2 \lesssim |\ln h| \left(  \sum_{K \in \mathcal{T}_h} \|v_h\|_{1, K}^2 + \sum_{e \in \mathcal{E}_h} \dfrac{1}{h_e} \|[v_h]\|_{0, e}^2   \right).
$$
Using the above inequality and Lemma \ref{Lem:L2leq|||}, it is easy to  obtain \eqref{Eqn:inversepromote}.
\end{proof}

For any $0 \leq m \leq r+1$, there exist an interpolation operator $\Pi_h: H^{r+1}(\Omega) \rightarrow V_h\cap C^0(\Omega)$ satisfies the following error estimation (See \cite{ScottZhang90:483}),
\begin{equation}\label{Eqn:u-Pihu}
\left(\sum_{K \in \mathcal{T}_h} \| w - \Pi_h w \|^2_{m, K} \right)^{1/2} \lesssim h^{r+1-m} \|w\|_{r+1},  \quad \forall w  \in H^{r+1}(\Omega).
\end{equation}

\begin{lemma}\label{Lem:u-Pihu|||h}
Suppose that the interpolation operator $\Pi_h$ is given in \eqref{Eqn:u-Pihu}, then for any $w \in H^{r+1}(\Omega) (r \geq 1)$, we have
\begin{eqnarray}\label{Eqn:u-Pihu|||h}
|\|w- \Pi_h w \||_h \lesssim h^r \|w\|_{r+1}.
\end{eqnarray}
\end{lemma}

\begin{proof}
Using the definition of $|\| \cdot \||_h$, trace inequality, $w- \Pi_h w \in C^0(\Omega)$, quasi-uniform assumption of $\mathcal{T}_h$ and \eqref{Eqn:u-Pihu}, we have
\begin{eqnarray}
\lefteqn{|\|w- \Pi_h w \||_h } \nonumber
\\
&=& \left( \sum_{K \in \mathcal{T}_h} \|\nabla (w- \Pi_h w)\|^2_{0, K}
+ \sum\limits_{e \in \mathcal{E}_h} \dfrac{h_e}{\lambda} \| \{\nabla (w- \Pi_h w) \} \|_{0, e}^2
 + \sum\limits_{e \in \mathcal{E}_h} \dfrac{\lambda}{h_e} \|[w- \Pi_h w]\|_{0, e}^2  \right)^{1/2}
\nonumber
%
%
\\ \nonumber
&\lesssim&\sum_{K \in \mathcal{T}_h} \|\nabla (w- \Pi_h w)\|_{0, K}
+ \sum\limits_{K \in \mathcal{T}_h} \dfrac{h_e^{1/2}}{\lambda^{1/2}} \left(h_K^{-1/2} \| \nabla (w- \Pi_h w) \|_{0, K}
 + h_K^{1/2} |w- \Pi_h w|_{2, K} \right)
%
%
\\ \nonumber
&\lesssim&\sum_{K \in \mathcal{T}_h} \|\nabla (w- \Pi_h w)\|_{0, K}
+ \sum\limits_{K \in \mathcal{T}_h} (\| \nabla (w- \Pi_h w) \|_{0, K} + h_e |w- \Pi_h w|_{2, K} )
%
\\ \nonumber
\\ \label{Eqn:|||u-Piu|||}
&\lesssim& h^r \|w\|_{r+1, \Omega},
\end{eqnarray}
which completes the proof.
\end{proof}

\subsubsection{Weak formulation}

The weak formulation of \eqref{Eqn:u} suitable for the DG method is to find $u \in H_0^1(\Omega) \cap H^{2}(\mathcal{T}_h)$ such that
\begin{equation}\label{Eqn:XR}
a_h(u, v)  = (f(u), v), \quad \forall\  v \in H_0^1(\Omega) \cap H^{2}(\mathcal{T}_h),
\end{equation}
where the bilinear form $a_h(\cdot, \cdot)$  is defined in \eqref{Eqn:ah}.

\begin{lemma}\label{Lem:uWW}
If $u \in H_0^1(\Omega) \cap H^{2}(\Omega)$ is the solution of \eqref{Eqn:u}, then $u$ satisfies \eqref{Eqn:XR}.
On the contrary, if $u \in H_0^1(\Omega) \cap H^{2}(\Omega)$ is a solution of \eqref{Eqn:XR}, then $u$ is the solution of \eqref{Eqn:u}.
\end{lemma}

\begin{proof}
Firstly we prove that if $u \in H_0^1(\Omega) \cap H^{2}(\Omega)$ is the solution of \eqref{Eqn:u}, then $u$ is the solution of \eqref{Eqn:XR}.

For any $v \in  H_0^1(\Omega) \cap H^{2}(\mathcal{T}_h)$, using Green formula,  Lemma \ref{Lem:int}, the smoothness of $u$ and the definition of $a_h(\cdot, \cdot)$, we have
\begin{eqnarray}\nonumber
(-\Delta u, v)
&=& \sum_{K \in \mathcal{T}_h} (\nabla u , \nabla v)_K - \sum_{K \in \mathcal{T}_h} \int_{\partial K} v \cdot \nabla u \cdot n \mathrm{d}s
\\ \nonumber
&=& \sum_{K \in \mathcal{T}_h} (\nabla u , \nabla v)_K - \sum_{e \in \mathcal{E}_h} \int_e \{\nabla u\} \cdot[v] \mathrm{d} s-\sum_{e \in \mathcal{E}_h^0} \int_e \{v\}[\nabla u] \mathrm{d} s
\\ \nonumber
&=& \sum_{K \in \mathcal{T}_h} (\nabla u , \nabla v)_K - \sum_{e \in \mathcal{E}_h} \int_e \{\nabla u\} \cdot[v] \mathrm{d} s
\\ \label{Eqn:deltauv}
&=& a_h(u, v).
\end{eqnarray}
For any $v \in  H_0^1(\Omega) \cap H^{2}(\mathcal{T}_h)$, integrate with $v$ on both sides of the first equation of \eqref{Eqn:u}, we obtain
\begin{equation}\label{Eqn:deltauf}
(-\Delta u, v) = (f(u), v).
\end{equation}
Then combining \eqref{Eqn:deltauv} and \eqref{Eqn:deltauf}, we can see that $u$ is the solution of \eqref{Eqn:XR}.

We next show that if $u \in H_0^1(\Omega) \cap H^{2}(\Omega)$ is a solution of \eqref{Eqn:XR}, then $u$ is the solution of \eqref{Eqn:u}.

By the definition of $a_h(\cdot, \cdot)$, the smoothness of $u\in H_0^1(\Omega) \cap H^{2}(\Omega)$, we can obtain
$$
\sum_{K \in \mathcal{T}_h} (\nabla u, \nabla v)_K - \sum_{e \in \mathcal{E}_h} \int_e \{\nabla u \} \cdot [v] \mathrm{d}s = (f(u), v), \quad \forall v \in H_0^1(\Omega) \cap H^{2}(\mathcal{T}_h),
$$
Then using Green formula, Lemma \ref{Lem:int} and the smoothness of $u$ in the left hand side of the above equation,  we derive
$$
\sum_{K \in \mathcal{T}_h} (- \Delta u,  v)_K  = (f(u), v), \quad \forall v \in H_0^1(\Omega) \cap H^{2}(\mathcal{T}_h).
$$
Let $v\in C_0^\infty(\Omega) \subset H_0^1(\Omega) \cap H^{2}(\mathcal{T}_h)$ in the above equation, and using the additivity of the integral of $L^2$, the basic lemma of the variational method, we can obtain
$$
-\Delta u = f(u), \quad a.e.\ \mbox{in $\Omega$. }
$$
The boundary condition $u=0$ on $\partial \Omega$ is trival, since $ u \in H_0^1(\Omega) \cap H^{2}(\Omega)$.
\end{proof}

\subsubsection{Existence and uniqueness of finite element solution}
\label{Sec:well-posed}

We will use the Brouwer fixed point theorem to prove the
 well-posedness 
of the solution $u_h$ of the problem \eqref{Eqn:Varuh} in this subsection.

Similar to the idea in \cite{GudiNataraj08:233},  with solution $u$ given in problem \eqref{Eqn:u}, we consider the following problem:
Find $u_h \in V_h$ such that
\begin{eqnarray} \nonumber
\lefteqn{a_h(u-u_h, v_h) -(f_u(u)(u-u_h), v_h) }\\  \label{Eqn:au-uh}
&&= (f(u), v_h) - (f(u_h), v_h) - (f_u(u)(u-u_h), v_h),~~\forall~v_h\in V_h.
\end{eqnarray}
Equations \eqref{Eqn:Varuh} and \eqref{Eqn:XR} imply that the problem \eqref{Eqn:Varuh} and \eqref{Eqn:au-uh} are equivalent.

To prove the existence of $u_h$ in problem \eqref{Eqn:au-uh}, for any $v_h \in V_h$, we define operator $\Phi_h: V_h \rightarrow V_h$ by
\begin{align}\nonumber
& a_h(u-\Phi_h(w_h), v_h) -(f_u(u)(u-\Phi_h(w_h)), v_h)\\ \label{Eqn:au-Phiwh}
& = (f(u), v_h) - (f(w_h), v_h) - (f_u(u)(u-w_h), v_h).
\end{align}
It can be proved that $\Phi_h$ is well defined, i.e., for any $w_h \in V_h$, there exist a unique $\Phi_h(w_h) \in V_h$ satisfies \eqref{Eqn:au-Phiwh}.
In fact, we can rewrite \eqref{Eqn:au-Phiwh} as
\begin{equation}\label{Eqn:au-PhiwhLS}
b_h(\Phi(w_h), v_h) = (F(u, w_h), v_h), \quad \forall v_h \in V_h,
\end{equation}
where
\begin{align*}
b_h(\Phi(w_h), v_h) &= a_h(\Phi_h(w_h), v_h) -(f_u(u)\Phi_h(w_h), v_h),
\\
(F(u, w_h), v_h) &=a_h(u, v_h)   - (f(u), v_h) + (f(w_h), v_h) - (f_u(u)w_h, v_h).
\end{align*}

The coerciveness of $a_h(\cdot, \cdot)$ and Assumption \ref{Assf} show that the bilinear form $b_h(\cdot, \cdot)$ is coercive, namely,
\begin{equation}\label{Eqn:positive}
|\| v_h \||_h^2  \lesssim a_h(v_h,v_h) - (f_u(u)v_h, v_h) = b_h(v_h, v_h).
\end{equation}
Using the continuity of $a_h(\cdot, \cdot)$, Assumption \ref{Assf}, Cauchy inequality and Lemma \ref{Lem:L2leq|||}, for any $w_h, v_h \in V_h$, we have
\begin{eqnarray} \nonumber
b_h(w_h,v_h) &=& a_h(w_h,v_h) - (f_u(u)w_h, v_h)
\\ \nonumber
&\lesssim& |\| w_h \||_h  \cdot |\| v_h\||_h +  \|w_h\|_0 \cdot  \|v_h\|_0
\\ \label{Eqn:continuity}
&\lesssim&  |\| w_h \||_h  \cdot |\| v_h\||_h.
\end{eqnarray}
The formula \eqref{Eqn:continuity} is the continuity of bilinear form $b_h(\cdot, \cdot)$.
In addition, by the continuity of the bilinear form $a_h(\cdot, \cdot)$, Assumption \ref{Assf}, Cauchy inequality and Lemma \ref{Lem:L2leq|||}, we know that $(F(u, w_h), \cdot)$ is a continuous linear functional defined on $V_h$ when $u$ and $w_h$ are given.
Therefore, according to the Lax-Milgram theorem, for any given $w_h \in V_h$, the problem \eqref{Eqn:au-PhiwhLS} has a unique solution, which is written as $\Phi_h(w_h)$.
Thus the operator $\Phi_h$ is well defined.

For a given solution $u$ of problem \eqref{Eqn:u}, we define a space
\begin{equation}\label{Eqn:Bdelta}
B_h=\{ v_h \in V_h : |\| \Pi_h u - v_h\||_h \leq \delta_h  \},
\end{equation}
where the interpolation operator $\Pi_h$ is defined in \eqref{Eqn:u-Pihu},
$\delta_h = C_1 |\| \Pi_h u - u\||_h$, $C_1 >1$ is a constant which can be sufficiently large and does not depend on the mesh size.

We can prove that the space $B_h$ is a non-empty compact convex subset.
In fact, since $v_h = \Pi_h u  \in  B_h$, $B_h$ is a non-empty space.
By using the triangle inequality, it is also easy to verify that the space $B_h$ is convex.
In a finite dimensional space, we can prove that $B_h$ is compact by proving that $B_h$ is bounded and closed.
Let's first prove that $B_h$ is bounded.
According to the triangle inequality and the definition of the space $B_h$, we have
$$
|\|v_h\||_h - |\|\Pi_h u \||_h \leq |\| \Pi_h u - v_h\||_h \leq \delta_h,  \quad \forall v_h \in B_h.
$$
Thus
$$
|\|v_h\||_h \leq \delta_h + |\|\Pi_h u \||_h, \quad \forall\  v_h \in V_h.
$$
And then we can prove that $B_h$ is closed.
Let $\{v_n\}$ be the Cauchy sequence in $B_h$, then there exist $v \in V_h$ such that $v_n \rightarrow v \ (n \rightarrow \infty)$.
Therefore, for any $q>0$, there exist $v_{n_0}$ such that $|\| v-v_{n_0}\||_h \leq q$, then we have
$$
|\|v- \Pi_hu\||_h \leq |\|v- v_{n_0}\||_h + |\|v_{n_0}- \Pi_hu\||_h \leq q + \delta_h.
$$
So $|\|v- \Pi_hu\||_h \leq \delta_h$, i.e. $v \in B_h$.
In conclusion, the space $B_h$ is a non-empty compact convex subset.

To prove that the solution of problem \eqref{Eqn:Varuh} is unique, several lemmas need to be proved first.
\begin{lemma}\label{Lem:PhiBsubsetB}
Suppose the operator $\Phi_h$ and the functional space $B_h$ are given by \eqref{Eqn:au-Phiwh} and \eqref{Eqn:Bdelta}, respectively, and $r\geq d/2  $, then when $h$ is small enough, we have
$
\Phi_h(B_h) \subset B_h.
$
\end{lemma}
\begin{proof}
For any $w_h \in B_h$, by \eqref{Eqn:positive}, \eqref{Eqn:au-Phiwh} and \eqref{Eqn:continuity}, we can obtain
\begin{eqnarray}
\lefteqn{|\| \Pi_h u - \Phi_h(w_h)\||_h^2} \nonumber
\\ \nonumber
&\lesssim& a_h(\Pi_h u - \Phi_h(w_h),\Pi_h u - \Phi_h(w_h))
\\ \nonumber
&&- (f_u(u)(\Pi_h u - \Phi_h(w_h)), \Pi_h u - \Phi_h(w_h))
%
\\ \nonumber
&=&
a_h(u - \Phi_h(w_h),\Pi_h u - \Phi_h(w_h)) - (f_u(u)(u - \Phi_h(w_h)), \Pi_h u - \Phi_h(w_h))
\\ \nonumber
&&+ a_h(\Pi_h u - u,\Pi_h u - \Phi_h(w_h)) - (f_u(u)(\Pi_h u - u), \Pi_h u - \Phi_h(w_h))
\\ \nonumber
&\lesssim& (f(u), \Pi_h u - \Phi_h(w_h)) - (f(w_h), \Pi_h u - \Phi_h(w_h))
\\ \label{Eqn:Pihu-PhihwhLS}
&&  - (f_u(u)(u-w_h), \Pi_h u - \Phi_h(w_h)) + |\| \Pi_h u - u \||_h \cdot |\| \Pi_h u - \Phi_h(w_h) \||_h. 
\end{eqnarray}

Using Taylor expansion, it is obtained that
\begin{equation}\label{Eqn:TLwh1}
 f(w_h) = f(u) + f_u(u)(w_h-u) + f_{uu}(\theta_1)(w_h-u)^2,
\end{equation}
where $\theta_1$ is between $w_h$ and $u$.

Substituting \eqref{Eqn:TLwh1} into \eqref{Eqn:Pihu-PhihwhLS}, and using Assumption \ref{Assf}, the definition of $\delta_h$ in \eqref{Eqn:Bdelta}, \eqref{Eqn:L2leq|||}, Lemma \ref{Lem:inversepromote} and \eqref{Eqn:Bdelta}, we obtain
\begin{eqnarray}\nonumber
\lefteqn{|\| \Pi_h u - \Phi_h(w_h)\||_h^2}
\\ \nonumber
&\lesssim& |-(f_{uu}(\theta_1)(w_h-u)^2, \Pi_h u - \Phi_h(w_h))| +  |\| \Pi_h u - u \||_h \cdot |\| \Pi_h u - \Phi_h(w_h) \||_h
\\ \nonumber
&\lesssim& \|w_h - u \|_0^2 \cdot \|\Pi_h u - \Phi_h(w_h)\|_{0, \infty} + \dfrac{1}{C_1} \delta_h |\| \Pi_h u - \Phi_h(w_h) \||_h
%
\\ \nonumber
&\lesssim& \|w_h - u \|_0 \cdot |\|w_h - u \||_h \cdot  (\ln h)^{1/2} |\|\Pi_h u - \Phi_h(w_h)\||_{h} +  \dfrac{1}{C_1} \delta_h |\| \Pi_h u - \Phi_h(w_h) \||_h
%
\\ \nonumber
&\lesssim& (\ln h)^{1/2} \|w_h - u \|_0 \cdot ( |\|w_h - \Pi_hu \||_h + |\|\Pi_hu - u \||_h) \cdot |\|\Pi_h u - \Phi_h(w_h)\||_h
\\ \nonumber
&& +  \dfrac{1}{C_1} \delta_h |\| \Pi_h u - \Phi_h(w_h) \||_h
\\  \nonumber
&\lesssim& (\ln h)^{1/2} \|w_h - u \|_0 \cdot \left( \delta_h  +  \dfrac{1}{C_1} \delta_h \right) \cdot |\|\Pi_h u - \Phi_h(w_h)\||_h
\\ \label{|||Piu-Phiwh|||}
&&+  \dfrac{1}{C_1} \delta_h |\| \Pi_h u - \Phi_h(w_h) \||_h.
\end{eqnarray}

Using triangle inequality, Lemma \ref{Lem:L2leq|||}, interpolation error estimation \eqref{Eqn:u-Pihu}, \eqref{Eqn:Bdelta} and \eqref{Eqn:u-Pihu|||h}, we have
\begin{eqnarray}
\| w_h-u\|_{0} &\leq & \| w_h- \Pi_h u \|_{0} + \| \Pi_h u -u\|_{0} \nonumber
\\ \nonumber
&\lesssim& |\|w_h - \Pi_h u \||_h + h^{r+1}\|u\|_{r+1}
%
\\    \nonumber
&\lesssim& C_1 |\|u - \Pi_h u \||_h + h^{r+1}\|u\|_{r+1}
%
\\    \nonumber
&\lesssim&  C_1 h^{r} \|u\|_{r+1} + h^{r+1}\|u\|_{r+1}
%
\\  \label{Eqn:wh-uleqhh}
&\lesssim& C_1  h^{r} \|u\|_{r+1}.
\end{eqnarray}

Substituting \eqref{Eqn:wh-uleqhh} into \eqref{|||Piu-Phiwh|||}, we obtain
\begin{eqnarray*}
|\| \Pi_h u - \Phi_h(w_h)\||_h \lesssim C_1 h^{r} (\ln h)^{1/2
}  \cdot \left( \delta_h  + \dfrac{1}{C_1} \delta_h \right)  +  \dfrac{1}{C_1} \delta_h.
\end{eqnarray*}
Assume that the constant of the above $\lesssim$ is $C$, then we have
$$
|\| \Pi_h u - \Phi_h(w_h)\||_h \leq C\left[  C_1 h^{r} (\ln h)^{1/2
}  \cdot \left( 1 + \dfrac{1}{C_1}  \right)  +  \dfrac{1}{C_1}  \right]\delta_h.
$$
It is easy to prove that $\Phi_h (w_h) \subset B_h$ by proving $ C\left[  C_1 h^{r} (\ln h)^{1/2} \cdot \left( 1 + 1/C_1  \right)  +  1/C_1  \right] \leq 1$.
In fact, we can first take $C_1$ sufficiently large such that $C/C_1 <1$,
and then take $h$ sufficiently small such that $ C\left[  C_1 h^{r} (\ln h)^{1/2} \cdot \left( 1 + 1/C_1  \right)  +  1/C_1  \right] \leq 1$.
\end{proof}

\begin{lemma}\label{Lem:Phicontinuity}
Assume that the operator $\Phi_h$ and function space $B_h$ are defined in \eqref{Eqn:au-Phiwh} and \eqref{Eqn:Bdelta}, respectively, then the operator $\Phi_h$ is continuous in $B_h$.
\end{lemma}

\begin{proof}
For any $w_1, w_2 \in B_h$, by \eqref{Eqn:positive} and \eqref{Eqn:au-Phiwh}, we have
\allowdisplaybreaks
\begin{eqnarray} \nonumber
\lefteqn{|\| \Phi_h(w_2) - \Phi_h(w_1) \||_h^2}
\\ \nonumber
&\lesssim& a_h(\Phi_h(w_2) - \Phi_h(w_1), \Phi_h(w_2) - \Phi_h(w_1))
\\ \nonumber
&&- (f_u(u)(\Phi_h(w_2) - \Phi_h(w_1)), \Phi_h(w_2) - \Phi_h(w_1))
%
\\ \nonumber
&=& - a_h(u - \Phi_h(w_2), \Phi_h(w_2) - \Phi_h(w_1)) +  (f_u(u)(u - \Phi_h(w_2)), \Phi_h(w_2) - \Phi_h(w_1))
\\ \nonumber
&& + a_h(u - \Phi_h(w_1), \Phi_h(w_2) - \Phi_h(w_1)) -  (f_u(u)(u - \Phi_h(w_1)), \Phi_h(w_2) - \Phi_h(w_1))
\\ \nonumber
&=& - \left[ (f(u) -f(w_2) -f_u(u)(u-w_2), \Phi_h(w_2) - \Phi_h(w_1)) ) \right]
\\  \nonumber
&&+ \left[ (f(u) - f(w_1) - f_u(u)(u-w_1) , \Phi_h(w_2) - \Phi_h(w_1))  \right]
%
\\ \nonumber
&=& (f(w_2) - f(w_1), \Phi_h(w_2) - \Phi_h(w_1))
\\  \label{Eqn:PhihwLS}
&& + (f_u(u)(w_1 - w_2), \Phi_h(w_2) - \Phi_h(w_1)).
\end{eqnarray}

Using Taylor expansion, it is obtained that
\begin{eqnarray}\label{Eqn:fw2fw1TL}
f(w_2)  &=& f(w_1)+  f_u(w_1 + \kappa_1 (w_2 - w_1))(w_2 - w_1),
\\ \label{Eqn:fuufuwTL}
f_u(u) &=& f_u(w_1 + \kappa_1 (w_2 - w_1)) + f_{uu}(\theta_2) (u-w_1 - \kappa_1 (w_2 - w_1)),
\end{eqnarray}
where $\kappa_1 \in (0, 1)$, $\theta_2$ is between $u$ and $w_1 + \kappa_1 (w_2 - w_1)$.

Using \eqref{Eqn:PhihwLS}, Taylor expansion \eqref{Eqn:fw2fw1TL} and \eqref{Eqn:fuufuwTL}, Cauchy inequality, inverse inequality, \eqref{Eqn:L2leq|||} and \eqref{Eqn:wh-uleqhh}, we have
\begin{eqnarray*}
\lefteqn{|\| \Phi_h(w_2) - \Phi_h(w_1) \||_h^2}
\\
&=& (f_u(w_1 + \kappa_1 (w_2 - w_1)) (w_2 - w_1), \Phi_h(w_2) - \Phi_h(w_1))
\\
&& - (f_u(u)(w_2 - w_1), \Phi_h(w_2) - \Phi_h(w_1))
%
\\
&=& - (f_{uu}(\theta_2) (u-w_1 - \theta_2 (w_2 - w_1)) \cdot (w_2 - w_1), \Phi_h(w_2) - \Phi_h(w_1))
%
\\
&\lesssim& \|w_2-w_1\|_{0, \infty} \cdot \| u-w_1 - \theta_2 (w_2 - w_1)\|_{0} \cdot  \|\Phi_h(w_2) - \Phi_h(w_1)\|_0
\\
&\lesssim&  h^{-d/2}\| w_2-w_1 \|_0  \cdot (\theta_2 \| u -w_2\|_{0} + (1- \theta_2) \| u -w_1\|_{0} ) \cdot |\| \Phi_h(w_2) - \Phi_h(w_1) \||_h
%
\\
&\lesssim&
 h^{-d/2} |\| w_2-w_1 \||_h  \cdot C_1 h^r \|u\|_{r+1} \cdot |\| \Phi_h(w_2) - \Phi_h(w_1) \||_h
%
\\
&\lesssim&
C_1h^{r -d/2} \|u\|_{r+1} \cdot |\| w_2-w_1 \||_h  \cdot |\| \Phi_h(w_2) - \Phi_h(w_1) \||_h.
\end{eqnarray*}
From the above equation and the arbitrariness of $w_1, w_2 \in B_h$, we know that $\Phi_h$ is continuous.
\end{proof}

\begin{theorem}\label{Eqn:uhexistunique}
Assume $r \geq d/2$, then when $h$ is small enough, the discrete variational problem \eqref{Eqn:Varuh} exists a unique solution $u_h$.
\end{theorem}

\begin{proof}
Making use of Lemmas \ref{Lem:PhiBsubsetB} and \ref{Lem:Phicontinuity}, Brouwer fixed point theorem, we know that the discrete variational problem \eqref{Eqn:Varuh} exists at least one solution $u_h$ in space $B_h$.

Next, we prove that the discrete variational problem \eqref{Eqn:Varuh} has only a unique solution in the finite element space $V_h$.

Assume $u_1$ and $u_2$ are both the solutoin of \eqref{Eqn:Varuh}, using the Taylor expansion, we have
$$
a_h(u_1 - u_2, v_h) = (f(u_1) - f(u_2), v_h) = (f_u (\theta)(u_1 - u_2), v_h),
$$
where $\theta$ is between $u_1$ and $u_2$.

Taking $v_h = u_1 - u_2$ in the above equation, and using \eqref{Eqn:positive}, we have
$$
|\| u_1 - u_2\||_h^2 \lesssim a_h(u_1 - u_2, u_1 - u_2) - (f_u (\theta)(u_1 - u_2), u_1 - u_2) = 0.
$$
Therefore $u_1 = u_2$.
Because $B_h \subset V_h$, the discrete variational problem \eqref{Eqn:Varuh} exists unique solution $u_h$ in finite element space $V_h$, and $u_h \in B_h$.
\end{proof}

\section{Optimal priori error estimates}\setcounter{equation}{0}\label{Sec:error}

In this section, the optimal priori error estimates of the discrete variational problem \eqref{Eqn:Varuh} will be proved.

\subsection{Projection operator $P_h$ and its properties}\label{Sec:Ph}

A new projection operator needs to be introduced firstly before we provide the errors estimates.

Set $s>\frac{1}{2}$, we define a projection operator $P_h: H^{1+s}(\mathcal{T}_h) \rightarrow V_h$ by
\begin{equation}\label{Eqn:Ph}
a_h(P_h w, v_h) = a_h(w, v_h), \quad \forall\  w \in H^{1+s}(\mathcal{T}_h) , v_h \in V_h.
\end{equation}

By using Lemma \ref{Lem:ahleqgeq} and Lax-Milgram theorem, we know that $P_h$ is well defined, namely, for any $w\in H^{1+s}(\mathcal{T}_h)$, there exist a unique $P_h w \in V_h$ that satisfies \eqref{Eqn:Ph}.
And the projection operator $P_h$ satisfies the following estimate.
\begin{lemma}\label{Lem:u-Phu}
For any $w\in H^{r+1}(\Omega) ( r\geq 1 )$, we have
\begin{equation}\label{Eqn:u-Phu}
\| w - P_hw \|_{0, \Omega} \lesssim h^{r+1} |w|_{r+1, \Omega}.
\end{equation}
\end{lemma}

\begin{proof}
Since $w \in H^{r+1}(r \geq 1)$, there exist a function $F \in L^2(\Omega)$ such that $-\Delta w=F$.
Similar to the derivation of \eqref{Eqn:deltauv}, we can prove that $a_h(w, v_h) = (F, v_h)$.
By the definition of $P_h$, the coerciveness and continuity of $a_h(\cdot, \cdot)$, we can prove that $P_h w$ is the unique solution of the following problem:
$$
\left\{\begin{aligned}
&\text{Find}\ w_h \in V_h\ \text{such that}
\\
&a_h(w_h, v_h) = (F, v_h), \quad \forall v_h \in V_h.
\end{aligned}\right.
$$
Then, the finite element solution $P_h w$ and $w$ have the following error estimates (detailed proof can be found in Section 2.3 of \cite{ZhangT12Book})
$$
\|w-P_hw\|_0 \leq C h^{r+1} \|w\|_{r+1}.
$$
\end{proof}

In order to give the approximation of the projection operator $P_h$, the following two preparatory lemmas should be  introduced.

We define function space
\begin{equation}\label{Eqn:Bprime}
B^\prime = \{ v_h \in V_h : \| P_h u - v_h \|_0 \leq \eta \},
\end{equation}
where the projection operator $P_h$ is defined in \eqref{Eqn:Ph}, $\eta = C_0 \| P_h u - u\|_0$, $C_0 > 1$ is a constant that can be sufficiently large and does not depend on the mesh size.
Similar to the space $B_h$ in \eqref{Eqn:Bdelta}, we can also prove that $B^\prime$ is a non-empty compact convex subset.

\begin{lemma}\label{Lem:PhiBpsubsetBp}
Assume that the operator $\Phi_h$ and function space $B^\prime$ are given in \eqref{Eqn:au-Phiwh} and \eqref{Eqn:Bprime}, respectively, and $r \geq d/2$, then when $h$ is small enough, we have $\Phi_h(B^\prime) \subset B^\prime$.
\end{lemma}

\begin{proof}
For any $w_h \in B^\prime$, by Lemma \ref{Lem:L2leq|||}, \eqref{Eqn:positive}, \eqref{Eqn:Ph}, \eqref{Eqn:au-Phiwh} and Assumption \ref{Eqn:Assf}, we have
\allowdisplaybreaks
\begin{eqnarray}\nonumber
\lefteqn{\| P_h u - \Phi_h(w_h)\|_0^2  \leq |\| P_h u - \Phi_h(w_h)\||_h^2
}
\\ \nonumber
&\lesssim& a_h(P_h u - \Phi_h(w_h),P_h u - \Phi_h(w_h)) - (f_u(u)(P_h u - \Phi_h(w_h)), P_h u - \Phi_h(w_h))
%
\\ \nonumber
&=&
a_h(u - \Phi_h(w_h),P_h u - \Phi_h(w_h)) - (f_u(u)(u - \Phi_h(w_h)), P_h u - \Phi_h(w_h))
\\ \nonumber
&&+ a_h(P_h u - u,P_h u - \Phi_h(w_h)) - (f_u(u)(P_h u - u), P_h u - \Phi_h(w_h))
%
\\ \nonumber
&=&
a_h(u - \Phi_h(w_h),P_h u - \Phi_h(w_h)) - (f_u(u)(u - \Phi_h(w_h)), P_h u - \Phi_h(w_h))
\\ \nonumber
&&- (f_u(u)(P_h u - u), P_h u - \Phi_h(w_h)) 
\\ \nonumber
&\lesssim& (f(u), P_h u - \Phi_h(w_h)) - (f(w_h), P_h u - \Phi_h(w_h))
\\ \nonumber 
&& - (f_u(u)(u-w_h), P_h u - \Phi_h(w_h)) + \| P_h u - u\|_0 \cdot \| P_h u - \Phi_h(w_h) \|_0.
\end{eqnarray}

From this, and similar to the proof of Lemma \ref{Lem:PhiBsubsetB}, the conclusion can be obtained.
\end{proof}

Similar to the proof of Lemma \ref{Lem:Phicontinuity}, the following lemma can also be obtained for space $B^\prime$, which is omitted here.
\begin{lemma}\label{Lem:Phi2continuity}
Assume that the operator $\Phi_h$ and the function space $B^\prime$ are defined in \eqref{Eqn:au-Phiwh} and \eqref{Eqn:Bprime}, respectively, then the operator $\Phi_h$ is continuous in space $B^\prime$.
\end{lemma}

At the end of this subsection, by using Lemmas \ref{Lem:PhiBpsubsetBp} and \ref{Lem:Phi2continuity}, the following approximation of the projection operator $P_h$ can be obtained.
\begin{lemma}\label{Lem:uh-PhuL2}
Assume $u_h \in V_h$ and $u\in H_0^1(\Omega) \cap H^{r+1}(\Omega)$ are the solution of problem \eqref{Eqn:Varuh} and \eqref{Eqn:u}, respectively, $P_h$ is given in \eqref{Eqn:Ph}, then when $h$ is sufficiently small, we have
\begin{equation}\label{Eqn:uh-PhuL2}
\|u_h - P_h u \|_0 \leq  C_0 \|u - P_h u \|_0 \lesssim  h^{r+1} \|u\|_{r+1}.
\end{equation}
\end{lemma}

\begin{proof}
First of all, from Lemma \ref{Lem:PhiBpsubsetBp}, Lemma \ref{Lem:Phi2continuity} and Brouwer fixed point theorem, we know that  \eqref{Eqn:Varuh} exists at least one solution in $B^\prime$.
The solution $u_h$ of the problem \eqref{Eqn:Varuh} is unique in $V_h$, which have been proved in Section \ref{Sec:well-posed}.
Therefore, the solution $u_h$ of \eqref{Eqn:Varuh} also belong to $B^\prime$. Then according to the definition of $B^\prime$, the first inequality of \eqref{Eqn:uh-PhuL2} is true..

Secondly, the second inequality of \eqref{Eqn:uh-PhuL2} can be proved by using the projection error estimation \eqref{Eqn:u-Phu}.
\end{proof}

\subsection{Priori error analysis}\label{Sec:prioruh}

In this subsection,  we will give the error estimate between the finite element solution $u_h$ of the problem \eqref{Eqn:Varuh} and the solution $u$ of the problem \eqref{Eqn:u}.

\begin{lemma}\label{Lem:u-uhleqinf}
Assume $u$ and $u_h$ are the solutions of \eqref{Eqn:u} and \eqref{Eqn:Varuh}, respectively, then we have
\begin{equation}\label{Eqn:u-uhleqinf}
|\| u-u_h\||_h \lesssim \inf_{v_h \in V_h} |\| u-v_h\||_h + \sup_{w_h \in V_h} \dfrac{|a_h(u, w_h) - a_h(u_h, w_h)|}{|\| w_h\||_h}.
\end{equation}
\end{lemma}

\begin{proof}
Set $v_h \in V_h$. Using the coerciveness and continuity of $a_h(\cdot. \cdot)$, we obtain
\begin{eqnarray}\nonumber
|\| v_h - u_h\||_h^2 &\lesssim& a_h( v_h - u_h,  v_h - u_h)
%
\\ \nonumber
&=&  a_h( v_h - u,  v_h - u_h) + a_h( u - u_h,  v_h - u_h)
\\ \label{Eqn:vh-uh2}
&\lesssim& |\| v_h - u\||_h |\| v_h - u_h\||_h + a_h( u ,  v_h - u_h) - a_h( u_h,  v_h - u_h).  
\end{eqnarray}

Divide both sides of \eqref{Eqn:vh-uh2} by $|\| v_h - u_h\||_h$, we obatin
\begin{eqnarray}\nonumber
|\| v_h - u_h\||_h  &\lesssim& |\| v_h - u\||_h + \dfrac{a_h( u ,  v_h - u_h) - a_h( u_h,  v_h - u_h)}{|\| v_h - u_h\||_h}
\\ \nonumber 
&\lesssim&   |\| v_h - u\||_h + \sup\limits_{w_h \in V_h} \dfrac{|a_h( u ,  w_h) - a_h( u_h,  w_h)|}{|\| w_h\||_h}.
\end{eqnarray}

Using triangle inequality and above inequality, we have
\begin{eqnarray}\nonumber
|\| u-u_h\||_h  &\leq& |\| u-v_h\||_h  + |\| v_h-u_h\||_h
\\ \nonumber
&\lesssim& |\|  u - v_h\||_h + \sup\limits_{w_h \in V_h} \dfrac{|a_h( u ,  w_h) - a_h( u_h,  w_h)|}{|\| w_h\||_h}.
\end{eqnarray}
Further, using above inequality and the arbitrariness of $v_h \in V_h$, we complete the proof.
\end{proof}

The error estimate between the finite element solution $u_h$ and the solution $u$ of the problem \eqref{Eqn:u} under DG norm is given below.
\begin{theorem}\label{Lem:u-uh|||}
Assume $u$ and $u_h$ are the solutons of problem \eqref{Eqn:u} and \eqref{Eqn:Varuh}, respectively, then we have
\begin{equation}\label{Eqn:u-uh|||}
|\| u -u_h\||_h \lesssim h^r\|u\|_{r+1}.
\end{equation}
\end{theorem}

\begin{proof}
Using \eqref{Eqn:Varuh} and \eqref{Eqn:XR}, for any $w_h \in V_h$, we have
\begin{equation}\label{au-auhLS}
a_h(u, w_h) - a_h(u_h, w_h) = (f(u), w_h) - (f(u_h), w_h).
\end{equation}

Using Taylor expansion, it is obtained that
\begin{equation}\label{au-auhTL}
f(u)= f(u_h) + f_u(\theta_4)(u-u_h),
\end{equation}
where $\theta_4$ is between $u$ and $u_h$.

Substituting the Taylor expansion \eqref{au-auhTL} into \eqref{au-auhLS}, and using Cauchy inequality, Assumption \ref{Assf}, triangle inequality, \eqref{Eqn:Bprime}, \eqref{Eqn:L2leq|||} and Lemma \ref{Lem:uh-PhuL2}, we obatin
\begin{eqnarray} \nonumber
a_h(u, w_h) - a_h(u_h, w_h) &=& (f_u(\theta_4)(u-u_h), w_h)
\\ \nonumber
& \lesssim& \|u-u_h\|_0 \|w_h\|_0
%
\\ \nonumber
&\lesssim& (\|u-P_hu\|_0 + \|P_hu-u_h\|_0) |\|w_h\||_h
%
\\ \nonumber
&\lesssim& (\|u-P_hu\|_0 +  \|u-P_hu\|_0) |\|w_h\||_h
%
\\ \label{au-auh}
&\lesssim& h^{r+1}\|u\|_{r+1} |\| w_h\||_h.
\end{eqnarray}
Finally, \eqref{Eqn:u-uh|||} can be obtained by Lemmas \ref{Lem:u-uhleqinf}, \ref{Lem:u-Pihu|||h} and \eqref{au-auh}.
\end{proof}

The error estimate between the finite element solution $u_h$ and the solution $u$ of the problem \eqref{Eqn:u} under $L^2$ norm is given below.
\begin{theorem}\label{Lem:u-uhL2}
Assume $u \in H_0^1(\Omega) \cap H^{r+1}(\Omega)$ and $u_h \in V_h$ are the solutions of problem \eqref{Eqn:u} and \eqref{Eqn:Varuh}, respectively, then we have
\begin{equation}\label{Eqn:u-uhL2}
\| u -u_h\|_0 \lesssim h^{r+1} \|u\|_{r+1}.
\end{equation}
\end{theorem}

\begin{proof}
Using triangle, Lemma \ref{Lem:uh-PhuL2} and \eqref{Eqn:u-Phu}, we obtain
\begin{eqnarray*}
\|u-u_h\|_0 &\leq& \|u - P_hu\|_0 + \|P_hu - u_h\|_0
\\
&\lesssim& \|u - P_hu\|_0 + \|u - P_hu\|_0
\\
&\lesssim&  h^{r+1} \|u\|_{r+1}.
\end{eqnarray*}
\end{proof}

\section{Numerical experiments}\label{Cha:5}
\setcounter{equation}{0}

In this section, we report several numerical
experiments in two-dimensional to verify the optimal convergence
order of the DG scheme \eqref{Eqn:Varuh}.
We implemented these experiments using the open-source scientific computing platform FEniCS \cite{LoggMardal12Book} using programming language Python.

Our model problem is
\begin{align*}
-\Delta u + u^3 = g(x)~\mbox{in}~\Omega, ~~~~u = 0~\mbox{on}~\partial \Omega,
\end{align*}
where the computational domain $\Omega = (0, 1) \times (0, 1)$, the exact solution  $u = \sin(\pi x) \sin(\pi y)$, and  $g(x)$ can be obtained from the exact solution.




{
We design several tests to verify our error estimates.
}
Firstly, we get a mesh $\mathcal{T}_h$ by partitioning the $x-$ and $y-$axes into equally subintervals, and then dividing each square into two triangles, see Figure \ref{fig:mesh}(a) for example.
We choose $V_h$ as the piecewise linear finite element space, i.e. $r = 1$,
and investigate the accuracy of the DG method \eqref{Eqn:Varuh} with various penalty parameters $\lambda$ and mesh sizes $h$.

\begin{figure}[htbp]
 \centering
 \begin{tabular}{cccc}
  \includegraphics[width=0.30\textwidth]{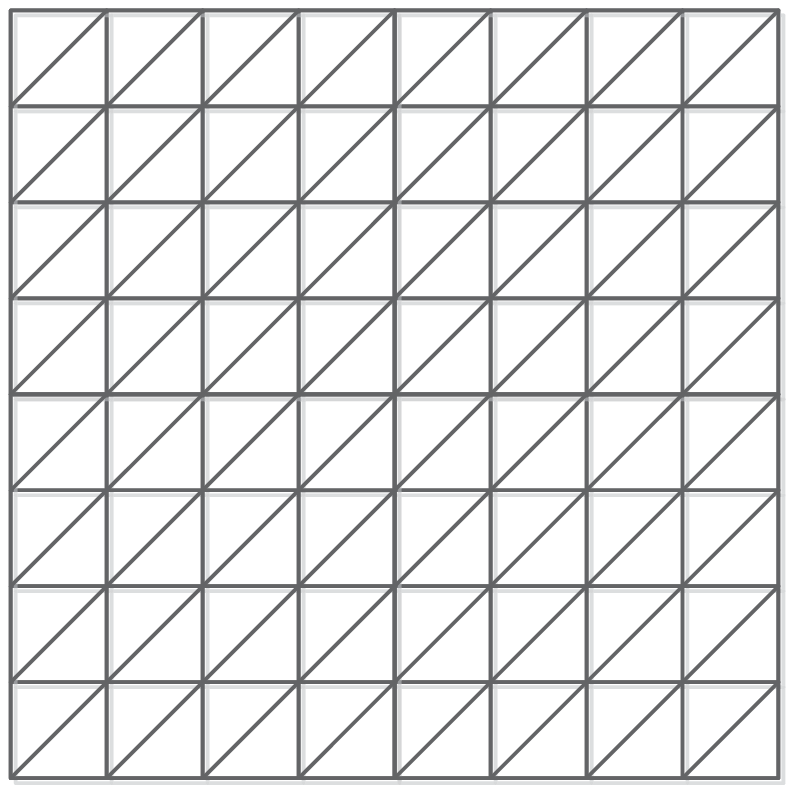} & \includegraphics[width=0.28\textwidth]{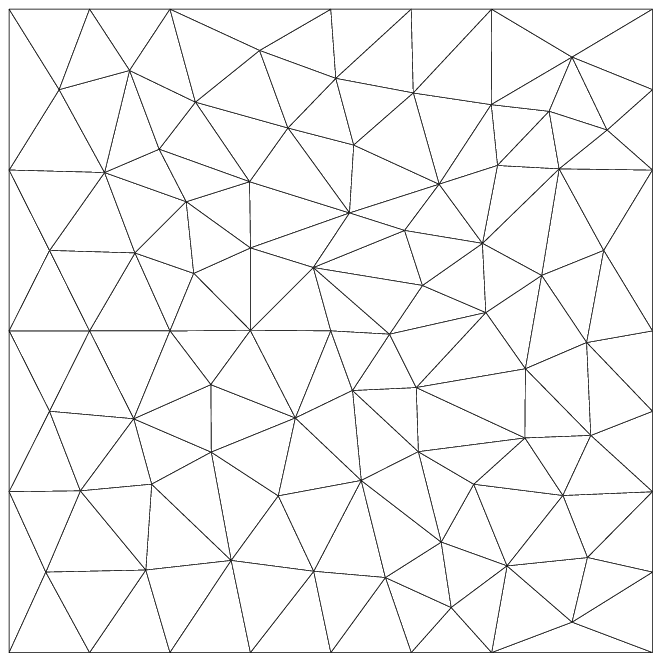}\\
  (a)   & (b)\\
 \end{tabular}
 \caption{Sturctured mesh and unstructured meshe.}
 \label{fig:mesh}
\end{figure}

\begin{table}[htbp]
\centering\caption{$r=1, \lambda = 10$}
\label{table-1}\vskip 0.1cm
\begin{tabular}{{|c|c|c|c|c|}}\hline
$h$     & $\|u-u_h\|_{0}$ &  order    & $\||u - u_h \||_h$    & order \\\hline
1/16    &1.10E-03   &--        &1.32E-01   &--    \\
1/32    &2.42E-04   &2.18      &6.44E-02   &1.04    \\
1/64    &5.61E-05   &2.11      &3.18E-02   &1.02    \\
1/128   &1.34E-05   &2.07      &1.58E-02   &1.01    \\\hline
\end{tabular}
\end{table}

\begin{table}[htbp]
\centering\caption{$r=1, \lambda = 100$}
\label{table-2}\vskip 0.1cm
\begin{tabular}{{|c|c|c|c|c|}}\hline
$h$     & $\|u-u_h\|_{0}$ &  order    & $\||u - u_h \||_h$    & order \\\hline
1/16    &1.03E-03   &--        &5.02E-02   &--     \\
1/32    &2.61E-04   &1.98      &2.41E-02   &1.06    \\
1/64    &6.56E-05   &1.99      &1.18E-02   &1.03    \\
1/128   &1.65E-05   &1.99      &5.85E-03   &1.01    \\\hline
\end{tabular}
\end{table}

\begin{table}[htbp]
\centering\caption{$r=1, \lambda = 1000$}
\label{table-3}\vskip 0.1cm
\begin{tabular}{{|c|c|c|c|c|}}\hline
$h$     & $\|u-u_h\|_{0}$ &  order    & $\||u - u_h \||_h$    & order \\\hline
1/16    &1.23E-03   &--        &1.72E-02   &--    \\
1/32    &3.09E-04   &1.99      &7.91E-03   &1.12    \\
1/64    &7.75E-05   &2.00      &3.83E-03   &1.05    \\
1/128   &1.94E-05   &2.00      &1.89E-03   &1.02    \\\hline
\end{tabular}
\end{table}

\begin{table}[htbp]
\centering\caption{$r=1, \lambda = 2000$}
\label{table-4}\vskip 0.1cm
\begin{tabular}{{|c|c|c|c|c|}}\hline
$h$     & $\|u-u_h\|_{0}$ &  order    & $\||u - u_h \||_h$    & order \\\hline
1/16    &1.25E-03   &--        &1.29E-02   &--    \\
1/32    &3.12E-04   &2.00      &5.70E-03   &1.18    \\
1/64    &7.82E-05   &2.00      &2.72E-03   &1.07    \\
1/128   &1.95E-05   &2.00      &1.34E-03   &1.02    \\\hline
\end{tabular}
\end{table}


Tables \ref{table-1}--\ref{table-4} show the $L^2$-norm and the DG norm of the error for different penalty parameters $\lambda = 10, 100, 1000$ and $2000$, respectively.
From this results, we observe that the convergence orders are optimal, as predicted by  \eqref{Eqn:u-uh|||} and  \eqref{Eqn:u-uhL2}.
And we can also notice that $\|u-u_h\|_{0}$ increases slightly as the penalty parameter $\lambda$ increases,
on the contrary, $|\|u-u_h\||_h$ decreases slightly as $\lambda$ increases.

{
Considering that the excessive penalty parameter will make the discontinuous finite element equations seriously ill-conditioned, therefore we might as well choose $\lambda = 100$ in the following experiments.
}

Some numerical results in Tables \ref{table-5}, \ref{table-6} are presented for the high order DG finite element spaces, i.e. $r=2, 3$,
and we can see that they are also optimal.
%

\begin{table}[htbp]
\centering\caption{$r=2, \lambda = 100$}
\label{table-5}\vskip 0.1cm
\begin{tabular}{{|c|c|c|c|c|}}\hline
$h$     & $\|u-u_h\|_{0}$ &  order    & $\||u - u_h \||_h$    & order \\\hline
1/16     &1.11E-05   &--        &3.05E-03  &--     \\
1/32     &1.25E-06   &3.15      &7.48E-04  &2.03    \\
1/64     &1.51E-07   &3.05      &1.85E-04  &2.02    \\
1/128    &1.85E-08   &3.03      &4.60E-05  &2.01    \\
\hline
\end{tabular}
\end{table}

\begin{table}[htbp]
\centering\caption{$r=3, \lambda = 100$}
\label{table-6}\vskip 0.1cm
\begin{tabular}{{|c|c|c|c|c|}}\hline
$h$     & $\|u-u_h\|_{0}$ &  order    & $\||u - u_h \||_h$    & order \\\hline
1/16    &9.30E-07   &--         &1.40E-04   &--       \\
1/32    &5.81E-08   &4.00       &1.75E-05   &3.00     \\
1/64    &3.51E-09   &4.05       &2.18E-06   &3.00     \\
1/128   &1.99E-10   &4.14       &2.73E-07   &3.00     \\\hline
\end{tabular}
\end{table}

At last, we consider the unstructured meshes (see Figure \ref{fig:mesh}(b) for example), and the corresponding numerical results are showd in  Tables \ref{table-7-1}--\ref{table-7-3} as follows.
We observe that all the the convergence orders are also optimal.
\begin{table}[htbp]
\centering\caption{Unstructured mesh, $r=1, \lambda = 100$}
\label{table-7-1}\vskip 0.1cm
\begin{tabular}{{|c|c|c|c|c|}}\hline
$h$     & $\|u-u_h\|_{0}$ &  order    & $\||u - u_h \||_h$    & order \\\hline
0.1       &1.89E-03   &--         &8.91E-02  &--       \\
0.05      &6.91E-04   &1.45       &5.08E-02  &0.81     \\
0.025     &1.74E-04   &1.99       &2.44E-02  &1.06     \\
0.0125    &4.43E-05   &1.97       &1.20E-02  &1.02     \\
\hline
\end{tabular}
\end{table}

\begin{table}[htbp]
\centering\caption{Unstructured mesh, $r=2, \lambda = 100$}
\label{table-7-2}\vskip 0.1cm
\begin{tabular}{{|c|c|c|c|c|}}\hline
$h$     & $\|u-u_h\|_{0}$ &  order    & $\||u - u_h \||_h$    & order \\\hline
0.1       &1.50E-04   &--          &1.17E-02  &--       \\
0.05      &1.63E-05   &3.20        &2.48E-03  &2.24      \\
0.025     &2.17E-06   &2.91        &6.46E-04  &1.94      \\
0.0125    &2.50E-07   &3.12        &1.51E-04  &2.09      \\
\hline
\end{tabular}
\end{table}

\begin{table}[htbp]
\centering\caption{Unstructured mesh, $r=3, \lambda = 100$}
\label{table-7-3}\vskip 0.1cm
\begin{tabular}{{|c|c|c|c|c|}}\hline
$h$     & $\|u-u_h\|_{0}$ &  order    & $\||u - u_h \||_h$    & order \\\hline
0.1       &8.68E-06   &--           &7.20E-04  &--       \\
0.05      &6.25E-07   &3.80         &8.95E-05  &3.01       \\
0.025     &3.94E-08   &3.99         &1.06E-05  &3.08       \\
0.0125    &2.27E-09   &4.12         &1.26E-06  &3.08       \\
\hline
\end{tabular}
\end{table}

\section*{Acknowledgements}

The first and second authors are  supported by the National Natural Science Foundation of China(Nos. 12071160, 11671159), the Guangdong Basic and Applied Basic Research Foundation (No. 2019A1515010724), the Characteristic Innovation Projects of Guangdong colleges and universities, China (No. 2018KTSCX044) and the General Project topic of Science and Technology in Guangzhou, China (No. 201904010117).
The third author  is also supported by China Postdoctoral Science Foundation (Grant No. 2019M652925).

%

 \end{document}